\documentclass[12pt]{article}
\usepackage{amsmath, amsfonts, amsthm, amssymb}
\usepackage{graphicx}
\usepackage{float}
\usepackage{verbatim}
\usepackage{color}
\numberwithin{equation}{section}

\newtheorem{thm}{Theorem}[section]
\newtheorem{lma}[thm]{Lemma}
\newtheorem{cor}[thm]{Corollary}

\newtheorem{prop}[thm]{Proposition}

\newcommand{\R}{\mathbb{R}}

\newcommand{\N}{\mathbb{N}}

\providecommand{\norm}[1]{\lVert#1\rVert}

\renewcommand{\geq}{\geqslant}
\renewcommand{\leq}{\leqslant}
\renewcommand{\epsilon}{\varepsilon}

\renewcommand{\i}{\mathbf{i}}

\allowdisplaybreaks

\title{Dimensions of equilibrium measures\\ on a class of planar  self-affine sets}

\author{Jonathan M. Fraser$^1$, Thomas Jordan$^2$ \& Natalia Jurga$^3$ \\ \\
\emph{$^1$School of Mathematics and Statistics,}\\ \emph{University of St Andrews, St Andrews, KY16 9SS, UK}\\ \emph{E-mail}: jmf32@st-andrews.ac.uk
\\ \\
\emph{$^2$School of Mathematics,}\\ \emph{University of Bristol, Bristol, BS8 1SN, UK}\\ \emph{E-mail}:  thomas.jordan@bristol.ac.uk 
\\ \\
\emph{$^3$ Department of Mathematics,}\\ \emph{University of Surrey, Guildford, GU2 7XH, UK}\\ \emph{E-mail}:  N.Jurga@surrey.ac.uk 
 }

\begin{document}

\maketitle

\begin{abstract}
We study equilibrium measures (K\"{a}enm\"{a}ki measures) supported on self-affine sets generated by a finite collection of diagonal and anti-diagonal matrices acting on the plane and satisfying the strong separation property.  Our main result is  that such measures are exact dimensional and the dimension satisfies the Ledrappier-Young formula, which gives an explicit expression for the dimension in terms of the entropy and Lyapunov exponents as well as the dimension of a coordinate projection of the measure.  In particular, we do this by showing that the K\"aenm\"aki measure is equal to the sum of (the pushforwards) of two Gibbs measures on an associated subshift of finite type. \\


\end{abstract}

\section{Introduction} \label{intro}

\subsection{ The Ledrappier-Young formula and dimensions of measures}

Computing the dimensions of measures invariant under a non-conformal dynamical system is a challenging problem, which has attracted a lot of attention in the literature in recent years.  The `Ledrappier-Young formula' refers to a dimension formula originating with \cite{ledrappieryoung1, ledrappieryoung2}, which gives the exact dimension of a measure in a non-conformal setting in terms of entropy, Lyapunov exponents and dimensions of projected measures.  The measures originally considered by Ledrappier-Young were invariant measures for $C^2$ diffeomorphisms, but the formula has shown up in various contexts since their original work.  B\'ar\'any and K\"aenm\"aki \cite{baranykaenmaki} proved that any self-affine measure in the plane satisfies the Ledrappier-Young formula and it is another interesting problem to consider self-affine measures in higher dimensional space. For self-affine measures in higher dimensional space whose Lyapunov exponents are all distinct (which is a generic condition in the sense of \cite[Lemma 6.2]{baranykaenmakikoivusalo}) this problem was  solved  by \cite[Theorem 2.3]{baranykaenmaki}.  Another direction is to consider more general measures supported on self-affine sets, noting that self-affine measures correspond to Bernoulli measures on the associated  shift space.  With some additional assumptions B\'ar\'any and K\"aenm\"aki \cite{baranykaenmaki} extended their results to include certain classes of quasi-Bernoulli measures, i.e., measures where one has the Bernoulli property up to a uniform constant. 

 In this paper we consider a natural class of measures which are \emph{not} quasi-Bernoulli and prove that they satisfy the appropriate version of the Ledrappier-Young formula.  One of our main tools, which may be of independent interest, is that, even though our measures are not quasi-Bernoulli, we show they are equal to the sum of (the pushforwards) of two quasi-Bernoulli measures on an associated subshift of finite type. This enables us to reduce the study of measures on equilibrium states on these self-affine sets to the study of Gibbs measures on graph directed self-similar systems. There is a paper in preparation by De-Jun Feng which shows that all projections of ergodic measures onto self-affine sets are exact dimensional. In our specific setting we are able to connect the dimension to a one-dimension graph-directed system which gives more information about the dimension (e.g. Corollaries \ref{cor2.2} and \ref{cor2.3}).

The \emph{upper} and \emph{lower local dimensions} of a Borel probability measure $\nu$ at a point $x$ in its support are defined by
\[
\overline{\dim}_{\mathrm{loc}}(\nu,x) \  = \ \limsup_{r \to 0} \frac{\log \nu\big(B(x,r) \big)}{\log r}  \quad \text{and} \quad  \underline{\dim}_{\mathrm{loc}}(\nu,x) \  = \ \liminf_{r \to 0} \frac{\log \nu\big(B(x,r) \big)}{\log r}.
\]
If the upper and lower local dimensions coincide, we call the common value the \emph{local dimension} and denote it by $\dim_{\mathrm{loc}}(\nu,x)$.   The measure $\nu$ is \emph{exact dimensional} if there exists a constant $\alpha$ such that the local dimension exists and equals $\alpha$ at $\nu$ almost all points. The lower Hausdorff dimension of $\nu$ is defined by
\[
\underline{\dim}_{\mathrm{H}} \nu = \mathrm{essinf}_{x \sim \nu} \underline{\dim}_{\mathrm{loc}} (\nu, x)
\]
and the upper packing dimension is 
\[
\overline{\dim}_{\mathrm{P}} \nu = \mathrm{esssup}_{x \sim \nu} \overline{\dim}_{\mathrm{loc}} (\nu, x).
\]
The lower packing and upper Hausdorff dimensions are defined similarly but, in particular, if $\nu$ is exact dimensional then all of these definitions coincide with the exact dimension $\alpha$.

\subsection{Our class of planar self-affine sets} \label{class}

Here we introduce our class of self-affine sets and the class of measures we study.  These self-affine sets were introduced in \cite{fraser_box} and were designed as a natural extension of the self-affine carpets introduced by Bedford-McMullen and developed by several others \cite{bedford,mcmullen,baranski,lalley-gatz,fengwang}.  In general, \emph{carpets} refer to planar self-affine sets generated by diagonal matrices.  The key difference in the sets we consider here is the presence of anti-diagonal matrices, which causes the system to be irreducible and fail the totally dominated splitting condition.  This class of self-affine sets was  also recently considered by Morris \cite{morris}, who derived an explicit formula for the pressure allowing very straightforward calculation of the affinity and box dimensions.

Let $\mathcal{I} = \{1, \dots, d\}$ be a finite alphabet and $\{S_i\}_{i \in \mathcal{I}}$ a finite collection of affine maps acting on the plane such that the associated linear parts are contracting non-singular $2 \times 2$ real-valued matrices with non-negative entries which are all either diagonal or anti-diagonal and we assume the collection contains at least one of each.  
In particular, we order the maps in the following way: let $S_i=A_i + t_i$ where
$$A_i= 
\begin{bmatrix}
a_i & 0 \\
0 & b_i
\end{bmatrix}$$
for $i<l$
and
$$A_i= 
\begin{bmatrix}
0& a_i \\
b_i & 0
\end{bmatrix}$$
for $i \geq l$
where $l-1$ is the number of diagonal matrices in the IFS, that is, $1<l \leq |\mathcal{I}|=d$. We will also assume that for some $1\leq i\leq l-1$ we have that $a_i\neq b_i$ (so our system is not self-similar).  

We may assume for convenience that each $S_i$ maps the unit square into itself.  It is well-known that there exists a unique non-empty compact set $F \subseteq [0,1]^2$ satisfying:
\[
F \ = \ \bigcup_{i \in \mathcal{I}} S_i(F)
\]
which is the \emph{self-affine set} corresponding to the iterated function system (IFS) $\{S_i\}_{i \in \mathcal{I}}$.

For $k \in \mathbb{N}$ let $\mathcal{I}^k$ denote words of length $k$ over $\mathcal{I}$ and let
\[
\mathcal{I}^* = \bigcup_{k \in \mathbb{N}} \mathcal{I}^k
\]
be the set of all finite words over $\mathcal{I}$.  Let $\Sigma = \mathcal{I}^\mathbb{N}$ be the set of all infinite words over $\mathcal{I}$ and equip this space with the product topology and any compatible metric and let $\sigma: \Sigma \to \Sigma$ be the left shift.  For $\i=(i_1, i_2, \dots ) \in \Sigma$, write $\i|_k = (i_1, i_2, \dots, i_k) \in \mathcal{I}^k$ for the restriction of $\i$ to its first $k$ symbols.  For $\i = (i_1, i_2, \dots, i_k) \in \mathcal{I}^k$, write
\[
S_\i \ = \ S_{i_1} \circ \cdots \circ S_{i_k}.
\]
For $\i \in \mathcal{I}^*$ write $1 > \alpha_1(\i) \geq \alpha_2(\i) >0$ for the singular values of the linear part of the affine map $S_\i $. 

The shift space $(\Sigma, \sigma)$ gives a useful symbolic coding of $F$, via the following correspondence.  Let $\Pi: \Sigma \to \mathbb{R}^2$ be defined by
\[
\Pi(\i) \ =  \bigcap_{k=1}^{\infty} S_{\i|_k} ([0,1]^2)
\]
where $\i=(i_1, i_2, \dots ) \in \Sigma$.  It is straightforward to see that $F = \Pi(\Sigma)$ and generally $\Pi$ need not be injective.

\subsection{Entropy and  Lyapunov exponents}

Let $m$ be any shift ergodic Borel probability measure on $\Sigma$ and let $\mu = m \circ \Pi^{-1}$ be the push forward of $m$ onto the fractal $F$.  The support of $\mu$ is contained in $F$, although in general the inclusion can be strict.  We say $\mu$ is an ergodic measure for the self-affine set $F$ and such measures are our main object of study in this paper.

The following classical result follows from the subadditive ergodic theorem. 

\begin{prop}[Sub-additive ergodic theorem] \label{oseledets}
There exist positive constants  (Lyapunov exponents)  $0<\chi_1(\mu) \leq \chi_2(\mu)  $ such that for $m$ almost all $\i \in \Sigma$ we have
\[
\chi_1(\mu) \ =  \ - \lim_{n \to \infty} \frac{1}{n} \log \alpha_1(\i|_n)
\]
and
\[
\chi_2(\mu)  \ =  \  -  \lim_{n \to \infty} \frac{1}{n} \log \alpha_2(\i|_n).
\]
\end{prop}
In fact in the setting of this paper we will be able to express these values as integrals of functions defined on a suitable chosen subshift of finite type.  It can be shown that if $m$ is a fully supported Bernoulli measure on $F$ in the setting of section \ref{class} (in particular the iterated function system is irreducible meaning that there is no one-dimensional linear subspace of $\mathbb{R}^2$ that is preserved by all of the matrices $A_i$) then the Lyapunov exponents will be equal \cite[Theorem 13]{morris2} and it will be straightforward to calculate the dimension. In Section 7 of \cite{MS}  they calculate the dimension of the attractor in several cases by looking at the dimension of $\sigma^n$-invariant Bernoulli measures without full support. Our aim will be to look at equilibrium states and show they are exact dimensional and that in several cases they give an ergodic measure of maximal dimension.

Another key ingredient will be entropy, this time provided by the Shannon-McMillan-Breiman theorem.

\begin{prop}[Shannon-McMillan-Breiman] \label{SMB}
There exists a non-negative constant  (the entropy)  $h(\mu) $ such that for $m$ almost all $\i \in \Sigma$ we have
\[
h(\mu) \ =  \  -  \lim_{n \to \infty} \frac{1}{n} \log m([\i|_n]).
\]
\end{prop}

\subsection{Description of equilibrium states}

For $s \in (0,2]$, let $\phi^s : \mathcal{I}^* \to \mathbb{R}^+$ be Falconer's submultiplicative singular value function, see \cite{affine}, defined by
\[
\phi^s(\textbf{i}) \ = \ \left\{ \begin{array}{cc}
\alpha_1(\textbf{i})^s & s \in (0,1) \\
\alpha_1(\textbf{i}) \, \alpha_2(\textbf{i})^{s-1} & s \in [1, 2]\\
\end{array} \right.
\]
This allows the subadditive pressure to be defined by
$$P(s)=\lim_{n\to\infty}\frac{1}{n}\log\left(\sum_{\textbf{i} \in \{1,\ldots,d\}^n}\phi^s(\textbf{i} )\right).$$
Assuming that $P(2) \leq 0$, we define the \emph{affinity dimension} to be the value $s^{\ast}$ for which $P(s^{\ast})=0$. It is shown in \cite{morris} that in our setting the affinity dimension may be calculated as the spectral radius of a certain matrix. In the current paper we will use that it can be found in terms of the usual additive pressure of a suitable potential on a suitable subshift of finite type.

It is easy to see that the potentials $\phi^{s}$  are submultiplicative. We can also see that our iterated function system is irreducible and thus by the work in \cite{FK} and \cite{FL} it follows  that there exists a unique ergodic Borel probability measure $m^s$ and a universal constants $C_{0} \geq 1$ such that for all $\textbf{i} \in \mathcal I^*$ we have
\begin{equation}\label{Gibbs}
C_0^{-1} e^{-P(s)|\textbf{i}|}\, \phi^{s}([\textbf{i}])  \  \leq \  m^s([\textbf{i}]) \  \leq \ C_0e^{-P(s)|\textbf{i}|} \, \phi^{s}([\textbf{i}]) 
\end{equation}
where $|\textbf{i}|$ denotes the length of the string.  (Throughout this paper we will use `universal constant' to mean a constant which is independent of quantities which vary in our arguments, such as $n$ and $\mathbf{i}$, although it may depend on parameters which are fixed in the statements of our results, such as the maps in our iterated function system or  the value of $s$.) This follows from Proposition 1.2 in \cite{FK} when $s\in (0,1]$. When $1< s\leq 2$ the proof of Proposition 1.2 in \cite{FK} can be easily adapted.  See the comment directly below Question 3.1 in \cite{FK} and also \cite{FL}.  
Fix $s<2$ and let $\mu=\mu^s$ be the ergodic measure on $F$ corresponding to $m=m^s$. We call $m$ (and $\mu$) a K\"aenm\"aki measure, following \cite{kaenmaki} where such measures were first considered in this context.

We also say $m$ (and $\mu$) are \emph{quasi-Bernoulli measures} if there exists a universal constant $C \geq 1$ such that for all $\textbf{i}, \textbf{j} \in \mathcal I^*$ we have
\[
\frac{1}{C} \, m([\textbf{i}]) \,  m([\textbf{j}]) \  \leq \  m([\textbf{i}\textbf{j}]) \  \leq \ C \,  m([\textbf{i}]) \,   m([\textbf{j}]).
\]
Note that all standard Gibbs measures are quasi-Bernoulli but in general, the K\"aenm\"aki measure is \emph{not} quasi-Bernoulli, but it \emph{is} always submultiplicative when \eqref{Gibbs} is satisfied, in that the right hand side of the above always holds.   In particular in our case it will not be supermultiplicative (since we have fixed $s<2$) because of the presence of both anti-diagonal and diagonal matrices. In particular, using a similar argument to \cite[Lemma 3.5]{MS} we can consider $i$ where $$A_i= 
\begin{bmatrix}
a_i & 0 \\
0 & b_i
\end{bmatrix}$$
and $a_i\neq b_i$ and $j$ where
$$A_j= 
\begin{bmatrix}
0 & a_j \\
b_j&0
\end{bmatrix}.$$
In this case we have that for $n\in\N$
$$A_i^n(A_j)A_i^n=\begin{bmatrix}
0&a_i^na_jb_i^n\\
b_i^nb_ja_i^n&0
\end{bmatrix}$$
and so for $0<s<2$ we have
$$\lim_{n\to\infty}\frac{\phi^s(A_i^n(A_j)A_i^n)}{\phi^s(A_i^nA_j)\phi^s(A_i^n)}=0$$
and so if $s<2$ the K\"{a}enm\"{a}ki measure cannot be supermultiplicative and, in particular, cannot be quasi-Bernoulli.  In particular, this shows that there does not exist a H\"{o}lder continuous potential whose Gibbs measure is the  K\"aenm\"aki measure since if there did, then it would be quasi-Bernoulli by definition.

Note that in \cite[Definition 2.6]{baranyrams} some sufficient conditions which ensure the K\"{a}enm\"{a}ki measure \emph{is} quasi-Bernoulli are given.  Also, in \cite[Proposition 7.3]{baranykaenmakikoivusalo} some sufficient conditions which ensure the K\"{a}enm\"{a}ki measure fails to be quasi-Bernoulli are given, albeit in a different setting.

Consider the sub-shift of finite type on $\{1, \dots , 2d\}^{\mathbb{N}}$ corresponding to the transition matrix $A$ given by
\[
A(i,j) \ = \ \left\{ \begin{array}{cc}
1 & i \in \{1,\dots, l-1\} \cup \{d+l, \dots ,2d\}   \textnormal{ and }  j \leq d \\
1& i \in \{l,\dots, d+l-1\} \textnormal{ and } j>d \\
0 & \textnormal{ otherwise.} \\
\end{array} \right.
\]
Denote this by $\Sigma_A$ and define $\tau: \Sigma \to \Sigma_A$ by $\tau( \mathbf{i}) = \tau( i_1 i_2 \dots) = (\tau_1(\mathbf{i}) \tau_2(\mathbf{i}) \dots)$ where $\tau_1(\mathbf{i})=i_1$ and
\[
\tau_m(\mathbf{i}) \ = \ \left\{ \begin{array}{cc}
i_m & card\{1 \leq j \leq m-1 :i_j \geq l\} \textnormal{ even}\\
i_m +d & card\{1 \leq j \leq m-1 :i_j \geq l\} \textnormal{ odd.}\\
\end{array} \right.
\]
The purpose of this associated subshift of finite type is to precisely record at which times the orientation is preserved.  More precisely, $\tau_m(\mathbf{i})$ is in the `first half' of the double system if and only if the linear part of the map $S_{\mathbf{i}|_{m-1}}$ is a diagonal matrix.  Note that $\tau$ is not a surjection (but it is an injection) and the image of $\tau$ is the subset of $\Sigma_A$ consisting of sequences where the first digit is at most $d$. It will be convenient to introduce $\omega:\Sigma\to\Sigma_A$ which is the projection to the complement of $\tau(\Sigma)$, let $\omega: \Sigma \to \Sigma_A$ by $\omega( \mathbf{i}) = \omega( i_1 i_2 \dots) = (\omega_1(\mathbf{i}) \omega_2(\mathbf{i}) \dots)$ where $\omega_1(\mathbf{i})=i_1+d$ and
\[
\omega_m(\mathbf{i}) \ = \ \left\{ \begin{array}{cc}
i_m+d & card\{1 \leq j \leq m-1 :i_j \geq l\} \textnormal{ even}\\
i_m & card\{1 \leq j \leq m-1 :i_j \geq l\} \textnormal{ odd.}\\
\end{array} \right.
\]
We then have that $\Sigma_A=\tau(\Sigma)\cup\omega(\Sigma)$ where the union is disjoint.

Also, the subshift $\Sigma_A$ is mixing. To see this, observe that $A^2$ has all positive entries since the matrix
$$B^2=\begin{pmatrix} 1& 1&0&0\\ 0&0&1&1 \\0&0&1&1 \\1&1&0&0 \end{pmatrix}^2$$
has all positive entries. Note that $B$ is just a collapsed version of $A$ where the first row corresponds to $1 \leq i \leq l-1$, the second to $l \leq i \leq d$, the third to $d+1 \leq i \leq d+l-1$ and the last to $d+l \leq i \leq 2d$ (and similarly for columns too). In particular, since $\Sigma_A$ is mixing we know about the existence of unique Gibbs measures for the potentials which we define next.

We define locally constant potentials $f_{1,s}, f_{2,s}: \Sigma_A \to \mathbb{R}$ by
\[
f_{1,s}(\mathbf{i}) \ = \ \left\{ \begin{array}{cc}
s\log a_{i_1} & i_1 \leq d\\
s\log b_{i_1-d} & i_1 \geq d+1\\
\end{array} \right.
\]
and
\[
f_{2,s}(\mathbf{i}) \ = \ \left\{ \begin{array}{cc}
s\log b_{i_1} & i_1 \leq d\\
s\log a_{i_1-d} & i_1 \geq d+1\\
\end{array} \right.
\]
when $s \in (0,1)$ and similarly
\[
f_{1,s}(\mathbf{i}) \ = \ \left\{ \begin{array}{cc}
\log a_{i_1} +(s-1)\log b_{i_1} & i_1 \leq d\\
\log b_{i_1-d} +(s-1)\log a_{i_1-d} & i_1 \geq d+1\\
\end{array} \right.
\]
and
\[
f_{2,s}(\mathbf{i}) \ = \ \left\{ \begin{array}{cc}
\log b_{i_1}  +(s-1)\log a_{i_1}& i_1 \leq d\\
\log a_{i_1-d} +(s-1)\log b_{i_1-d} & i_1 \geq d+1\\
\end{array} \right.
\]
when $s \in [1,2]$. 

 We will denote the Gibbs measures on $\Sigma_A$ for these potentials by $m_1$ and $m_2$ respectively.  That is, the unique invariant Borel probability measures such that for any $\mathbf{i} \in \Sigma_A$ and $n \in \mathbb{N}$,
$$m_1([\mathbf{i}|_n]) \approx \exp(S_n f_{1,s}(\mathbf{i})- nP_A(f_{1,s}))$$
and
$$m_2([\mathbf{i}|_n]) \approx \exp(S_n f_{2,s}(\mathbf{i})- nP_A(f_{2,s}))$$
where $a \approx b$ means that $C^{-1} a \leq b \leq Ca$ for some universal constant $C$ and $S_nf(\mathbf{i})$ denotes the Birkhoff sum $f(\mathbf{i}) + \ldots + f(\sigma^{n-1} \mathbf{i})$.  Here $P_A(f_{1,s})$ and $P_A(f_{2,s})$ denotes the topological pressure on $\Sigma_A$ of $f_{1,s}$ and $f_{2,s}$. By symmetry we have that $P_A(f_{1,s})= P_A(f_{2,s})$. We  define a measure $\nu$ on $\Sigma$ by $\nu(E)= m_1(\tau(E))+m_2(\tau(E))$.   Recall that $\tau$ is an injection and so this indeed defines a measure.  Moreover, note that $\nu$ is a probability measure since $\nu(\Sigma)= m_1(\tau(\Sigma))+m_2(\tau(\Sigma)) =m_1(\tau(\Sigma))+m_1(\omega(\Sigma)) = m_1(\Sigma_A) $.  It follows from the definition of $\tau$, $m_1$ and the invariance of $m_1$ and $m_2$ that $\nu$ will be invariant.  Note that we emphasise the dependence of $f_{1,s}$ and $f_{2,s}$ on $s$, so we can use $f_{1,1}$ and $f_{2,1}$ to express the Lyapunov exponents, but for simplicity of exposition we deliberately suppress this dependence when writing the measures $m_1$, $m_2$, $\nu$ and of course $m$ and $\mu$. The measure $\nu$ has a similar structure to the equilibrium state studied in Proposition 6 of \cite{morris2} where a system with only anti-diagonal matrices is studied. The equilibrium state studied there could be put into this context but the subshift of finite type would not be mixing due to the lack of diagonal matrices.

 The following lemmas will show that $\nu$ is the K\"{a}enm\"{a}ki measure for $\phi^s$ and that the measures $m_1$ and $m_2$ can be used to find the Lyapunov exponents.
\begin{lma}\label{size}
We have that
$$\int f_{1,1} \textup{d} m_1=\int f_{2,1}\textup{ d} m_2\geq\int f_{2,1}\textup{d}m_1=\int f_{1,1}\textup{d}m_2$$
\end{lma}
\begin{proof}
The fact that $\int f_{1,1} \textup{d}m_1=\int f_{2,1}\textup dm_2$ and $\int f_{2,1}\textup{d}m_1=\int f_{1,1}\textup{d}m_2$ follows by symmetry (note that for $\mathbf{i}\in\Sigma$ and $s\in (0,2)$ we have that $f_{1,s}(\tau(\mathbf{i}))=f_{2,s}(\omega(\mathbf{i}))$ and $f_{2,s}(\tau(\mathbf{i}))=f_{1,s}(\omega(\mathbf{i}))$). 

For $m_1$ almost all $\mathbf{i}\in\Sigma_A$ we have that
$$\lim_{n\to\infty}\frac{S_nf_{1,1}(\mathbf{i})}{n}=\int f_{1,1} \textup{d}m_1=\int f_{2,1}\textup{d}m_2$$
and
$$\lim_{n\to\infty}\frac{S_nf_{2,1}(\mathbf{i})}{n}=\int f_{2,1} \textup{d}m_1=\int f_{1,1}\textup{d}m_2.$$
Thus if $\int f_{1,1} \textup{d}m_1<\int f_{2,1} \textup{d}m_1$ then for any $s\in (0,2)$
$\int f_{1,s} \textup{d}m_1<\int f_{2,s} \textup{d}m_1$. This means that using the Gibbs property and the fact that $P_A(f_{1,s})=P_A(f_{2,s})$ we have
$$\lim_{n\to\infty}\frac{m_2([i_1,\ldots,i_n])}{m_1([i_1,\ldots,i_n])} = \infty$$
which certainly cannot hold for $m_1$ almost all $\mathbf{i}$, see for example \cite[Theorem 2.12(1)]{mabook}, and so is a contradiction.
\end{proof}
\begin{lma}\label{key}
For all $\mathbf{i} \in \Sigma$ and $n \in \mathbb{N}$ we have that
$$\phi^s(\mathbf{i}|_n)= \max\{\exp S_n f_{1,s}(\tau(\mathbf{i})), \exp S_n f_{2,s}(\tau(\mathbf{i}))\}.$$
\end{lma}
\begin{proof}
This follows immediately from the definitions. In particular if $0<s \leq 1$ then $\exp S_n f_{1,s}(\tau(\mathbf{i}))$ corresponds to the length of the horizontal side of the rectangle $S_{\mathbf{i}|_n}([0,1]^2)$, raised to the power $s$ while $\exp S_n f_{2,s}(\tau(\mathbf{i}))$ corresponds to the length of the vertical side of the rectangle $S_{\mathbf{i}|_n}([0,1]^2)$, raised to the power $s$. Thus it is easy to see that $\phi^s(\mathbf{i}|_n)$ simply corresponds to the maximum of these two lengths. For $1 < s < 2$ analogous statements hold and the claim can be seen to also hold.
\end{proof}

\begin{cor} \label{equivalence}
There exists $C>0$ such that for all $\mathbf{i} \in \Sigma$ and $n \in \mathbb{N}$
\begin{eqnarray*}
C^{-1} \leq \frac{\nu([\mathbf{i}|_n])}{\phi^s(\mathbf{i}|_n) \exp(-nP_A(f_{1,s}))} \leq C
\end{eqnarray*}
In particular, $\nu=m$ is the unique K\"{a}enm\"{a}ki measure for $\phi^s$ and $P(s)=P_A(f_{1,s})=P_A(f_{2,s})$.
\end{cor}

\begin{proof}
It follows from the Gibbs property that 
$$\nu([\mathbf{i}|_n]) \approx \exp(-nP_A(f_{1,s})) \max \{\exp(S_nf_{1,s}(\tau(\mathbf{i}))), \exp(S_nf_{2,s}(\tau(\mathbf{i})))\}$$
and the first result now follows using Lemma \ref{key}.  Finally, by combining this and the Gibbs property for the K\"{a}enm\"{a}ki measure $m$ we get
\[
\frac{\nu([\mathbf{i}|_n])}{m([\mathbf{i}|_n]) } \approx \exp(n(P(s)-P_A(f_{1,s})))
\]
and since both $\nu$ and $m$ are probability measures we may conclude that $P_A(f_{1,s})=P(s)$. Finally since $\nu$ and $m$ are both invariant, it follows by the uniqueness of property \eqref{Gibbs} that $\nu=m$.
\end{proof}
We can also relate the Lyapunov exponents of $\mu$ (note that we now have that $\mu=\Pi\nu$ as $\nu=m$ )  to $m_1$ and $m_2$.
\begin{cor}\label{Lyap}
We have that
$$\chi_1(\mu)=-\int f_{1,1}\textup{d}m_1<\chi_2(\mu)=-\int f_{2,1}\textup{d}m_1$$
and for $\nu$-almost all $\mathbf{i}\in\Sigma$ we have that
$$\lim_{n\to\infty}\left( \frac{\alpha_2(\mathbf{i}|_n)}{\alpha_1(\mathbf{i}|_n)}\right)=0.$$
\end{cor}
\begin{proof}
We know that all $\mathbf{i} \in \Sigma$ satisfy
$$\alpha_1(\i|_n)=\max\{S_nf_{1,1}(\tau(\mathbf{i})),S_nf_{2,1}(\tau(\mathbf{i}))\}$$
and
$$\alpha_2(\i|_n)=\min\{S_nf_{1,1}(\tau(\mathbf{i})),S_nf_{2,1}(\tau(\mathbf{i}))\}.$$
Since we know by Lemma \ref{size} that $\int f_{1,1}\textup{d}m_1\geq\int f_{2,1}\textup{d}m_1$, and since $m_1\circ\tau \ll m$, it follows that for $m_1\circ\tau$ almost all $\mathbf{i} \in \Sigma$,
 $$\chi_1(\mu)=- \lim_{n \to \infty} \frac{1}{n} \log \alpha_1(\i|_n)   \ = \  -\int f_{1,1}\textup{d}m_1$$
 and
 \[
\chi_2(\mu)=-  \lim_{n \to \infty} \frac{1}{n} \log \alpha_2(\i|_n)    \ =  \   - \int f_{2,1}\textup{d}m_1.
\]
Similarly it can be shown that for $m_2\circ\tau$ almost all $\mathbf{i} \in \Sigma$,
 $$\chi_1(\mu)=- \lim_{n \to \infty} \frac{1}{n} \log \alpha_1(\i|_n)   \ = \  -\int f_{2,1}\textup{d}m_2$$
 and
 \[
\chi_2(\mu)=-  \lim_{n \to \infty} \frac{1}{n} \log \alpha_2(\i|_n)    \ =  \   - \int f_{1,1}\textup{d}m_2.
\]
To see why $\chi_1(\mu) < \chi_2(\mu)$, observe that if we had equality, then $m$ would be an equilibrium state for the additive potential $\mathbf{i} \mapsto \frac{s}{2} \log \text{Det}(A_{\mathbf{i}})$.  This would mean $m$ was quasi-Bernoulli and we have already observed that in our setting this is not the case.   Moreover, $\chi_1(\mu)< \chi_2(\mu)$ implies that for $m$-almost every $\mathbf{i} \in \Sigma$,
$$\lim_{n \to \infty} \log \left( \frac{\alpha_2(\mathbf{i}|_n)}{\alpha_1(\mathbf{i}|_n)}\right)^{\frac{1}{n}} <0$$
Thus, there exists $0<a<1$, $N \in \mathbb{N}$ such that for $n \geq N$,
$$\left( \frac{\alpha_2(\mathbf{i}|_n)}{\alpha_1(\mathbf{i}|_n)}\right)^{\frac{1}{n}} < a$$
In particular, for all $n \geq N$, $\frac{\alpha_2(\mathbf{i}|_n)}{\alpha_1(\mathbf{i}|_n)} \leq a^n$, and the conclusion follows.
\end{proof}

We can also conclude that in this case (again assuming that $P(2) \leq 0$) the affinity dimension is the value $s^{\ast}$ for which $P(s^{\ast})=0$, for details how to compute this value, see \cite{morris}.

The final piece of notation we introduce is $\mu_t= m_t \circ \tau \circ \Pi^{-1}$, the pushforward measure of $m_t \circ \tau$ onto $F$. We also need to consider the graph directed self-similar systems which correspond to the projections. We define  maps $g_i : [0,1] \to [0,1]$ by $g_i(x)=a_ix+\tau_i^{(1)}$ for $1\leq i\leq d$ and $g_i(x)=b_{i-d}x+\tau_{i-d}^{(2)}$ if $d+1\leq i\leq 2d$ and the standard natural projection $\overline{\Pi} : \Sigma_A \to\R$. We then have that for $\mathbf{i}\in\Sigma$
$$\pi_1(\Pi(\mathbf{i}))=\overline{\Pi}(\tau(\mathbf{i}))$$
and
$$\pi_2(\Pi(\mathbf{i}))=\overline{\Pi}(\omega(\mathbf{i})).$$
This allows us to deduce the following result.

\begin{prop} \label{projections2}
All the measures $\pi_1(\mu_1)$, $\pi_1(\mu_2)$, $\pi_2(\mu_1)$ and $\pi_2(\mu_2)$   are exact dimensional and we have that $\dim\pi_1(\mu_1)=\dim\pi_2(\mu_2)$ and $\dim\pi_2(\mu_1)=\dim\pi_1(\mu_2)$.  
\end{prop}

\begin{proof} 
Note that for any $\mathbf{i}\in\Sigma$ we have that $f_{1,s}(\tau(\mathbf{i}))=f_{2,s}(\omega(\mathbf{i}))$.

We claim that the measures $\overline{\Pi}(m_1)$ and $\overline{\Pi}(m_2)$ are exact dimensional. To see this, observe that these measures are supported on the set
\begin{eqnarray}
\bigcup_{\i \in \Sigma_A} \bigcap_{n=1}^{\infty} g_{\i|_n}([0,1]) \subset \bigcup_{\i \in \Sigma_{2d}} \bigcap_{n=1}^{\infty} g_{\i|_n}([0,1])
\label{gd ting}
\end{eqnarray}
where $\Sigma_{2d}$ denotes the full shift on $2d$ symbols and so the set on the right hand side of (\ref{gd ting}) is a self-similar set. 

Consider $m_t \circ \overline{\Pi}^{-1}$ as a measure on $\Sigma_{2d}$ and denote this by $m_t'$. Also let $\overline{\Pi}': \Sigma_{2d} \to \R$ be the projection onto the self-similar set
$$\overline{\Pi}'(\i)=  \bigcap_{n=1}^{\infty} g_{\i|_n}([0,1]).$$
Since $m_t$ is an invariant ergodic measure on $\Sigma$, it is straightforward to show that $m_t'$ is also invariant and ergodic for the full shift. Then by Theorem 2.8 in \cite {fenghu}, $\overline{\Pi}'(m_t')$ is exact dimensional and therefore by the relationship between the pairs $m_t, m_t'$ and $\overline{\Pi}, \overline{\Pi}'$ it immediately follows that $\overline{\Pi}(m_t)$ is also exact dimensional, completing the claim. 

However $\pi_1(\mu_1)$ and $\pi_2(\mu_2)$ are both absolutely continuous with respect to $\overline{\Pi}(m_1)$ and $\pi_1(\mu_2)$ and $\pi_2(\mu_1)$ are both absolutely continuous with respect to $\overline{\Pi}(m_2)$  and so the result follows. 
\end{proof}

\section{Results}

We now obtain our results by using the structure of the K\"aenm\"aki measure described in the previous section.  For convenience, we assume that the underlying IFS satisfies the \emph{strong separation property}, which means that for distinct $i , j \in \mathcal{I}$, we have $S_i(F) \cap S_j(F) = \emptyset$.

\begin{thm} \label{main}
Assume the self-affine set $F$ satisfies the strong separation property and let $\mu$ be any K\"aenm\"aki measure for $F$. Then $\mu$ is exact dimensional, with the exact dimension given by
\[
\dim \mu \ = \ \frac{h(\mu)}{\chi_2(\mu)} \, + \,  \frac{\chi_2(\mu)-\chi_1(\mu)}{\chi_2(\mu)}\dim\pi_1(\mu_1) .
\]
Thus $\mu$ satisfies the appropriate version of the Ledrappier-Young formula.
\end{thm}
We get the following corollary, which gives simpler formulae in the case where $\dim\pi_1(\mu_1)$ is what it is `expected to be'.
\begin{cor}\label{cor2.2}
If $\dim\pi_1(\mu_1)=\min\left\{\frac{h(\mu)}{\chi_1(\mu)},1\right\}$ then
$$\dim\mu=\left\{\begin{array}{lll}\frac{h(\mu)}{\chi_1(\mu)}&\text{ if }&h(\mu)\leq\chi_1(\mu)\\
1+\frac{h(\mu)-\chi_1(\mu)}{\chi_2(\mu)}&\text{ if }&h(\mu)>\chi_1(\mu).\end{array}\right.$$
\end{cor}
\begin{proof}
We first suppose that $\dim\pi_1(\mu_1)=\frac{h(\mu)}{\chi_1(\mu)}\leq1$. In this case we have
\begin{eqnarray*}
\dim\mu&=&\frac{h(\mu)}{\chi_2(\mu)} \, + \,  \frac{\chi_2(\mu)-\chi_1(\mu)}{\chi_2(\mu)} \dim\pi_1(\mu_1)\\
&=&\frac{h(\mu)}{\chi_2(\mu)} \, + \,  \frac{\chi_2(\mu)-\chi_1(\mu)}{\chi_2(\mu)}\frac{h(\mu)}{\chi_1(\mu)}\\
&=&\frac{h(\mu)}{\chi_1(\mu)}.
\end{eqnarray*}
On the other hand if $\dim\pi_1(\mu_1)=1<\frac{h(\mu)}{\chi_1(\mu)}$ then
\begin{eqnarray*}
\dim\mu&=&\frac{h(\mu)}{\chi_2(\mu)} \, + \,  \frac{\chi_2(\mu)-\chi_1(\mu)}{\chi_2(\mu)} \dim\pi_1(\mu_1)\\
&=&\frac{h(\mu)}{\chi_2(\mu)} \, + \,  \frac{\chi_2(\mu)-\chi_1(\mu)}{\chi_2(\mu)}\\
&=&1+\frac{h(\mu)-\chi_1(\mu)}{\chi_2(\mu)}
\end{eqnarray*}
completing the proof. 
\end{proof}

In \cite{MS} Morris and Shmerkin show in Proposition 7.2 that for a large class of such self-affine systems the dimension of the attractor is given by the expected affinity dimension. Here we are able to show that with a condition on the dimension of the projected measure, such systems have an ergodic measure of maximal dimension. However our measures on the projected system are not Bernoulli so, as yet, it is not possible to apply Hochman's results (\cite{hochman1},\cite{hochman2}) as is done in Propositions 7.2 and 7.3 in \cite{MS}. However it should be possible to adapt the necessary results of Hochman to the Markov and Gibbs case in which case the results of Proposition 7.2 and 7.3 in \cite{MS} would extend to include the existence of a measure of maximal dimension.

 However, transversality techniques can be used to study the dimension of the projected measure. In particular we can use the proof of Theorem 1 in \cite{KSS} to show that $\dim\pi_1(\mu)=\dim\pi_2(\mu)=\min\left\{\frac{h(\mu)}{\chi_1(\mu)},1\right\}$ for almost all translations $\mathbf{t} = (t_1, \dots, t_d) \in \mathbb{R}^{2d}$ satisfying a suitable condition on the norms of the matrices. To apply the results here we let $u_i=a_i$ if $1\leq i\leq l$ and let $u_i=a_i\max\{b_j:l+1\leq j\leq d\}$ if $l+1\leq i\leq d$. We also let $v_i=b_i$ if $1\leq i\leq l$ and let $v_i=b_i\max\{a_j:l+1\leq j\leq d\}$ if $l+1\leq i\leq d$. This means that in our setting the conditions needed in order to apply Lemma 3 in \cite{KSS} and hence show that $\dim\pi_1(\mu)=\dim\pi_2(\mu)=\min\left\{\frac{h(\mu)}{\chi_1(\mu)},1\right\}$ are that if $i\neq j$ then $u_i+u_j<1$ and $v_i+v_j<1$.

 This yields the following corollary which is a limited extension  of Theorem 1.9  in \cite{jordanetal} in this particular case, where the matrices are assumed to have norm smaller than $1/2$ (strong separation is not necessary in \cite{jordanetal}). 
 
  Note that the conditions for the following corollary are satisfied if $\norm{A_i}<\frac{1}{2}$ for $1\leq i\leq l$ and $\norm{A_i}<\frac{1}{\sqrt{2}}$ for $l+1\leq i\leq d$.
\begin{cor} \label{cor2.3}
Suppose that for $i,j\in\{1,\ldots,d\}$ with $i\neq j$ we have $u_i+u_j<1$ and $v_i+v_j<1$. We then have
if $\mu$ is a K\"{a}enm\"aki measure where $s \in (0,2)$, then for Lebesgue almost all $\mathbf{t} \in \R^{2d}$ where the strong separation property is satisfied we have that
$$\dim\mu=\left\{\begin{array}{lll}\frac{h(\mu)}{\chi_1(\mu)}&\text{ if }&h(\mu)\leq\chi_1(\mu)\\
1+\frac{h(\mu)-\chi_1(\mu)}{\chi_2(\mu)}&\text{ if }&h(\mu)>\chi_1(\mu).\end{array}\right.$$
\end{cor}

\section{Proof of Theorem \ref{main}}
In \cite{PU} Przytycki   and Urba\'{n}ski relate the dimension of a self-affine measure in two dimensions to the dimension of a self-similar measure in one dimension (in their case a Bernoulli convolution). We take the same approach, however we need to consider two measures in one dimension which in our case will be $\pi_1(\mu_1)$ and $\pi_2(\mu_2)$. Rather than being strictly self-similar, these are Gibbs measures on a graph directed self-similar iterated function system. This approach is also similar to one used by Falconer and Kempton in \cite{FK}.

\subsection{Some preliminary estimates}

It will be convenient to calculate the local dimension by measuring squares rather than balls in $\mathbb{R}^2$. To this end we introduce some notation. Let $Q_1(x,r)$ denote the \textit{one dimensional square} of side $r$ centred at $x$, given by $Q_1(x,r)=[x-\frac{r}{2}, x+\frac{r}{2}]$. For $\mathbf{x}=(x,y)$ let $Q_2(\mathbf{x}, r)$ denote the \textit{2-dimensional square} of side $r$ which is centred at $\mathbf{x}$, given by 
\[
Q_2(\mathbf{x}, r) \ = \ \left\{ (x', y') : |x-x'| \leq \frac{r}{2}  \text{ and } |y-y'|\leq \frac{r}{2} \right\}.
\]
Let $\mathbf{x}=(x,y) \in F$ with symbolic expansion $\mathbf{i} \in \Sigma$ and let $n \in \mathbb{N}$. Consider the cylinder $S_{\mathbf{i}|_n}([0,1]^2)$. Suppose the side lengths are distinct, so $\alpha_2(\mathbf{i}|_n) < \alpha_1(\mathbf{i}|_n)$. We shall call the longer side of $S_{\mathbf{i}|_n}([0,1]^2)$ the \textit{primary side}. Additionally we shall call the axis parallel to this side the \textit{primary axis} and denote the projection onto the primary axis by $\pi_p^{\mathbf{i},n}$. We may also call the direction of the primary axis the \textit{primary direction}. So that this is all well-defined even when $\alpha_1(\mathbf{i}|_n)=\alpha_2(\mathbf{i}|_n)$ we agree that in this scenario the primary axis is the $y$ axis. We denote the strip of all points inside $S_{\mathbf{i}|_n}([0,1]^2)$ that lie $\frac{r}{2}$-close to $\pi_p^{\mathbf{i},n}(\mathbf{x})$ in the primary direction by $B(\mathbf{x}, n, r)$, and refer to this as the \textit{primary strip}. We define the \textit{secondary projection} to be the primary projection if the linear part of $S_{\mathbf{i}|_n}$ preserves each co-ordinate axis (i.e., an even number of the linear parts $\left\{A_{i_1},\dots ,A_{i_n}\right\}$ are anti-diagonal matrices), and the other co-ordinate projection otherwise. We denote it by $\pi_s^{\mathbf{i},n}(\mathbf{x})$.

In order to prove the desired lower bound for the local dimension of the measure (which corresponds to finding an appropriate upper bound for the measure of any given primary strip), we  use sub-multiplicativity of the K\"aenm\"aki measure.  We can get an upper bound on the measure of a primary strip in terms of the product of the measure of an appropriate cylinder and an appropriate projected measure of the blow up of the strip.

\begin{lma}
Let $\mathbf{x}\in F$ with symbolic expansion $\mathbf{i}\in \Sigma$. For any $n \in \mathbb{N}$ and $r>0$ we have
$$ \mu(B(\mathbf{x}, n, r)) \ \leq  \ C m([\mathbf{i}|_n])\pi_s^{\mathbf{i},n}(\mu)\left(Q_1\left(\pi_s^{\mathbf{i},n}(\Pi(\sigma^n(\mathbf{i})), \frac{r}{\alpha_1(\mathbf{i}|_n)}\right)\right)$$  
\label{measure}
where $C \geq 1$ is the uniform constant giving submultiplicativity of $m$.
\end{lma}

\begin{proof}
Let
\[
\mathcal{J}  \  = \ \mathcal{J}(\mathbf{i}, n, r)  \ = \ \left\{ \textbf{j} \in \mathcal{I}^* \ : \ S_{\mathbf{i}|_n \textbf{j}}([0,1]^2) \subseteq B(\mathbf{x}, n, r) \text{ and } S_{\mathbf{i}|_n \textbf{j}^\dagger}([0,1]^2) \nsubseteq B(\mathbf{x}, n, r)   \right\}
\]
where $\textbf{j}^\dagger$ is $\textbf{j}$ with the last symbol removed.  Note that by our separation assumption the family of rectangles $\{S_{\mathbf{i}|_n \textbf{j}}([0,1]^2)\}_{\textbf{j} \in \mathcal{J} }$ are pairwise disjoint and exhaust $ B(\mathbf{x}, n, r)$ in measure.  Thus
\begin{eqnarray*}
\mu(B(\mathbf{x}, n, r))  &=& \sum_{\textbf{j} \in \mathcal{J} } m( [\mathbf{i}|_n \textbf{j} ] )  \\ \\
&\leq& C m( [ \mathbf{i}|_n ] )   \sum_{\textbf{j} \in \mathcal{J} } m(  [\textbf{j} ] )  \\ \\
&=&  Cm([\mathbf{i}|_n])\pi_s^{\mathbf{i},n}(\mu)\left(Q_1\left(\pi_s^{\mathbf{i},n}(\Pi(\sigma^n(\mathbf{i})), \frac{r}{\alpha_1(\mathbf{i}|_n)}\right)\right)
\end{eqnarray*}
where the last equality follows since the  family of rectangles $\{S_{\textbf{j}}([0,1]^2)\}_{\textbf{j} \in \mathcal{J} }$ are pairwise disjoint and exhaust $ S_{\mathbf{i}|_n}^{-1}B(\mathbf{x}, n, r)$ in measure.  Then, noting that $ S_{\mathbf{i}|_n}^{-1}B(\mathbf{x}, n, r)$ is a strip with one side of length 1 and the other of length  $r/\alpha_1(\mathbf{i}|_n)$ we have
\begin{eqnarray*}
\mu \left(S_{\mathbf{i}|_n}^{-1} \left(B(\mathbf{x}, n, r) \right) \right) &=&   \pi_s^{\mathbf{i},n}(\mu)\left(Q_1\left(\pi_s^{\mathbf{i},n}(\Pi(\sigma^n(\mathbf{i})), \frac{r}{\alpha_1(\mathbf{i}|_n)}\right)\right)
\end{eqnarray*}
as required.
\end{proof}

Next we prove an analogue of Lemma \ref{measure} giving an upper bound for the local dimension (so a lower bound for the measure). Here, given $\mathbf{i}\in\Sigma$, we will use that if $\tau(i_m)=i_m$ then $A_{i_1}\cdots A_{i_{m-1}}$ is diagonal whereas if $\tau(i_m)=i_m+d$ then $A_{i_1}\cdots A_{i_{m-1}}$ is anti-diagonal. This will allow us to get our analogue of Lemma \ref{measure} along a suitable subsequence (for typical points).

\begin{lma}
Let $\mathbf{x}\in F$ with symbolic expansion $\mathbf{i}\in \Sigma$, such that $\mathbf{i}$ satisfies that $l \leq i_n \leq |\mathcal{I}| $ for infinitely many $n$, i.e. infinitely many of the maps $S_{i_n}$ are `anti-diagonal'. (Observe that it is a direct consequence of the ergodic theorem that the set of such $\mathbf{i}$ is a set of full measure.) Let $n_k$ be any subsequence for which $1 \leq \tau_{n_k+1}(\mathbf{i}) \leq |\mathcal{I}|$ for all $k \in \mathbb{N}$.  Then for any $r>0$ and $k \in \mathbb{N}$,
\begin{eqnarray*}
\nu\circ \Pi^{-1}(B(\mathbf{x}, n_k, r)) &\geq& \mu_t(B(\mathbf{x}, n_k, r)) \\ 
  &\geq&  C  m_t \circ \tau([\mathbf{i}|_{n_k}])\pi_s^{\mathbf{i},n_k}(\mu_t)\left(Q_1\left(\pi_s^{\mathbf{i},n_k}(\Pi(\sigma^{n_k}(\mathbf{i})), \frac{r}{\alpha_1(\mathbf{i}|_{n_k})}\right)\right)
\end{eqnarray*} 
for $t=1, 2$  and where the constant $C$ is independent of $\mathbf{x}$, $k$, $r$, $t$. 
\label{measure2}
\end{lma}

\begin{proof}
Let $\mathcal{J}$ be as in the proof of Lemma \ref{measure}. Then, using the quasi-Bernoulli properties of $\mu_t$ we have
\begin{eqnarray*}
\mu_t(B(\mathbf{x}, n_k, r))  &=& \sum_{\textbf{j} \in \mathcal{J} } m_t \circ \tau( [\mathbf{i}|_{n_k} \textbf{j} ] )  \\ \\
&=& \sum_{\textbf{j} \in \mathcal{J} } m_t (\tau( [\mathbf{i}|_{n_k} ])\tau([\textbf{j} ] ))  \\ \\
&\geq& c m_t \circ \tau( [ \mathbf{i}|_{n_k} ] )   \sum_{\textbf{j} \in \mathcal{J} } m_t \circ \tau(  [\textbf{j} ] )  \\ \\
&=&  cm_t \circ \tau([\mathbf{i}|_{n_k}])\pi_s^{\mathbf{i},n_k}(\mu_t)\left(Q_1\left(\pi_s^{\mathbf{i},n_k}(\Pi(\sigma^{n_k}(\mathbf{i})), \frac{r}{\alpha_1(\mathbf{i}|_{n_k})}\right)\right)
\end{eqnarray*}
where the second equality holds because we are looking at times $n_k$ when $\mathbf{i}|_{n_k}$ has seen an even number of rotations, $c$ is a constant that comes from the quasi-Bernoulli properties of $m_1$ and $m_2$ and everything else follows by the same observations as in the proof of Lemma \ref{measure}.

Finally, the result follows because $\nu\circ \Pi^{-1}(B(x, n_k, r)) \geq\mu_t(B(x, n_k, r))$.
\end{proof}

Even though the measures $m_t \circ \tau$ are not generally ergodic, fortunately they share some ergodic properties with $m$.  In particular, we can control their ``Lyapunov exponents".

The following is essentially restating Corollary \ref{Lyap}
\begin{lma}
For $t=1,2$,  $m_t \circ \tau$ almost all $\mathbf{i} \in \Sigma$ satisfy
$$
 - \lim_{n \to \infty} \frac{1}{n} \log \alpha_1(\i|_n)   \ = \  \chi_1(\mu) 
<
-  \lim_{n \to \infty} \frac{1}{n} \log \alpha_2(\i|_n)    \ =  \    \chi_2(\mu).
$$
Thus for $\nu$-almost every $\mathbf{i} \in \Sigma$,
$$\frac{\alpha_2(\mathbf{i}|_n)}{\alpha_1(\mathbf{i}|_n)} \to 0.$$

\label{sequence1}
\end{lma}

The measures $m_t\circ \tau$ are not invariant and so measure theoretic entropy is not defined. However the following is true and is sufficient for our purposes.

\begin{lma}  \label{entropylemma}
For $t=1,2$, $m_t \circ \tau$ almost all $\mathbf{i} \in \Sigma$ satisfy
\[
 - \lim_{n \to \infty} \frac{1}{n}  \log  m_t \circ \tau([\mathbf{i}|_n])  = h(\mu).
\]
\end{lma}

\begin{proof}
Since $m_1$ and $m_2$ are distinct ergodic measures, they are mutually singular. Therefore $\mu_1$ is not absolutely continuous with respect to $\mu_2$. Thus by \cite[Theorem 2.1.2]{mabook}, 
$$\limsup_{r \to 0} \frac{\mu_2(B(x,r))}{\mu_1(B(x,r))} =0$$
for $\mu_1$ almost every $x$. Therefore
$$\limsup_{n \to \infty}\frac{m_2 \circ \tau([\i|_n])}{m_1\circ \tau([\i|_n])}=0$$
for $m_1 \circ \tau$ almost every $\i \in \Sigma$. In other words, for $m_1 \circ \tau$ almost every $\i$,  $m_1 \circ \tau([\i|_n])$ `dominates' $m_2 \circ \tau([\i|_n])$ for all large $n$ and 
$$h(\mu)=-\lim_{n \to \infty} \frac{1}{n} \log m([\i|_n])= -\lim_{n \to \infty}\frac{1}{n} \log m_1 \circ \tau([\i|_n])$$
for $m_1 \circ \tau$ almost all $\i$. By an analogous argument we obtain the same result for $m_2 \circ \tau$.
\end{proof}

Next, we obtain estimates for the projected measure of the blow up of a typical primary strip. We will let $s_1=\dim\pi_1(\mu_1)$.  The key point of the proof of this lemma is that an $m_1$-typical point $\tau(\mathbf{i})$ will regularly hit times $n$ when the $\mu$ measure of $B(\overline{\Pi}(\sigma^n(\tau(\textbf{i})),r)$ is sufficiently close to $r^{s_1}$ and the matrix $A_{i_1}\cdots A_{i_n}$ is diagonal (and the same for the measure $m_2$).

\begin{lma}
For $m$-almost every $\mathbf{i} \in \Sigma$ there exists a choice $t \in \{1,2\}$ and a strictly increasing sequence of positive integers  $n_k$ for which simultaneously $\mu_t(B(\mathbf{x}, n_k, r))$ satisfies the bound in Lemma \ref{measure2} for all $k \in \mathbb{N}$ and such that for all $\epsilon>0$ there exists $N_{\epsilon} \in \mathbb{N}$, such that for all $k \geq N_{\epsilon}$,
\begin{eqnarray}
&\,&  \hspace{-20mm}(s_1+ \epsilon) \log \frac{\alpha_2(\mathbf{i}|_{n_k})}{\alpha_1(\mathbf{i}|_{n_k})}\nonumber \\   & \leq&  \log \pi_s^{\mathbf{i}, n_k}(\mu)\left( Q_1\left( \pi_s^{\mathbf{i}, n_k}(\Pi(\sigma^{n_k}\mathbf{i}) , \frac{\alpha_2(\mathbf{i}|_{n_k})}{\alpha_1(\mathbf{i}|_{n_k})} \right) \right) \leq (s_1- \epsilon)\log \frac{\alpha_2(\mathbf{i}|_{n_k})}{\alpha_1(\mathbf{i}|_{n_k})} .
\label{ld bound}
\end{eqnarray}
Moreover, we can choose the sequence $n_k$ such that
\begin{eqnarray}
\lim_{k \to \infty} \frac{n_{k+1}}{n_k}  \   \to  \ 1.
\label{ratio}
\end{eqnarray}
\label{egorov}
\end{lma}

\begin{proof}
Recall that $\tau: \Sigma \to \tau\left(\Sigma\right)$ is one-to-one and thus has an inverse. With slight abuse of notation, for $\mathbf{i} \in \tau(\Sigma)$ denote  $\Pi(\mathbf{i})=\Pi(\tau^{-1}(\mathbf{i}))$ and $\alpha_r(\mathbf{i}|_n)=\alpha_r(\tau^{-1}(\mathbf{i})|_n)$ for $r=1,2$. We will show that for each $t=1,2$, for $m_t$-almost every $\mathbf{i} \in \tau(\Sigma)$, there exists a sequence $n_k$ such that (\ref{ld bound}) holds for $\mu_t$. Then the result will follow because the union of the pre-images under $\tau$ of these full measure sets for $m_1$ and $m_2$ have full $m$-measure. 

Fix $t \in \{ 1, 2\}$ and define the function $h_{t}^{(l)}: \tau(\Sigma) \to \mathbb{R}$ by
$$h_{t}^{(l)}(\mathbf{i})= \frac{ \log \pi_t(\mu)\left( Q_1\left( \pi_t(\Pi(\mathbf{i})), \frac{1}{l}\right)\right)}{\log \frac{1}{l}}$$
By Proposition \ref{projections2} each of $\pi_t(\mu_t)$ are exact dimensional with dimension $s_1$ and combining this with Theorem 2.12 from \cite{mabook} we deduce that for $m_t$-almost every $\mathbf{i} \in \tau(\Sigma)$
 $$\lim_{l \to \infty} h_{t}^{(l)}(\mathbf{i})=s_1$$ 
for $m_t$-almost every $\mathbf{i}\in \tau(\Sigma)$.
By Egorov's theorem, there exists a set $G_{t} \subset \tau(\Sigma)$ with measure $m_t(G_{t})\geq m_t(\tau(\Sigma))/2>0$, for which $h_{t}^{(n)}$  converges uniformly to $s_1$ for all $\mathbf{i} \in G_{t}$. In particular this means that for all $\epsilon >0$, there exists $L_{\epsilon} \in \mathbb{N}$ such that for $l \geq L_{\epsilon}$,
$$s_1- \epsilon \leq \frac{ \log \pi_t(\mu)\left( Q_1\left( \pi_t(\Pi(\mathbf{i})) , \frac{1}{l} \right) \right)}{\log \frac{1}{l}} \leq s_1+ \epsilon $$
for all $\mathbf{i} \in G_{t}$.  Moreover, by the Birkhoff Ergodic Theorem,
$$\lim_{n \to \infty} \frac{1}{n}\sum_{k=0}^{n-1} \textbf{1}_{G_{t}}(\sigma^{k}\mathbf{i}) \ =  \ \int \textbf{1}_{G_{t}} \, dm_t>0$$
for $m_t$-almost every $\mathbf{i} \in \tau(\Sigma)$. In other words, for $m_t$-almost every $\mathbf{i} \in \tau(\Sigma)$ we have that $\sigma^{n} \mathbf{i} \in G_{t}$ with  frequency greater than $0$.  Therefore, for such a fixed $\mathbf{i} \in \tau(\Sigma)$, we can choose  $n_{k}$ to be the subsequence of positive integers  such that $\sigma^{n_{k}}\mathbf{i} \in G_{t}$ for all $k \in \mathbb{N}$. By Lemma \ref{sequence1} we can choose $N_{\epsilon} \in \mathbb{N}$ such that for all $n \geq N_{\epsilon}$, 
$$\frac{\alpha_2(\mathbf{i}|_n)}{\alpha_1(\mathbf{i}|_n)}<\frac{1}{L_{\epsilon}}$$
Therefore, for $k \geq N_{\epsilon}$, we have 
$$(s_1+ \epsilon) \log \frac{\alpha_2(\mathbf{i}|_{n_{k}})}{\alpha_1(\mathbf{i}|_{n_{k}})}\leq  \log \pi_t(\mu)\left( Q_1\left( \pi_t(\Pi(\sigma^{n_{k}}\mathbf{i}) , \frac{\alpha_2(\mathbf{i}|_{n_{k}})}{\alpha_1(\mathbf{i}|_{n_{k}})} \right) \right) \leq (s_1- \epsilon)\log \frac{\alpha_2(\mathbf{i}|_{n_{k}})}{\alpha_1(\mathbf{i}|_{n_{k}})}. $$

 We now need to show that for $m_t$ almost all such $\mathbf{i}$ and large enough $k$, we have $\pi_s^{\mathbf{i}, n_k}= \pi_p^{\mathbf{i}, n_k}= \pi_t$, and (\ref{ld bound}) follows.  The fact that $\mathbf{i}|_{n_k}$ is in the diagonal case implies that $\pi_s^{\mathbf{i}, n_k}= \pi_p^{\mathbf{i}, n_k}$.  Moreover it follows by Lemmas \ref{Lyap} and \ref{size} that, for $k$ sufficiently large, for $m_t$ almost all $\mathbf{i}$
\[
S_{n_k}f_{t,1}(\mathbf{i})> S_{n_k}f_{t',1}(\mathbf{i}) \qquad (t' \neq t)
\]
Which in turn implies that the longer side of the rectangle $S_{\mathbf{i}|_{n_k}}([0,1]^2)$ is the horizontal side if $t=1$ and vertical side if $t=2$.  Therefore $ \pi_p^{\mathbf{i}, n_k}= \pi_t$ as claimed.

It only remains to prove that  $\frac{n_{k+1}}{n_k} \to 1$ as $k \to \infty$. To see this let $S_{n_{k}} = \sum_{r=0}^{n_{k}} \chi_{G_{t}} (\sigma^r(\mathbf{i}))$. Then the ergodic theorem tells us that $\lim_{k\to \infty} \frac{S_{n_{k}}}{n_{k}}=m_t(G_t)>0$. Moreover, clearly $\frac{S_{n_{k+1}}}{n_{k+1}} = \frac{S_{n_{k}} +1}{n_{k+1}}$. Now,
\begin{eqnarray*}
 \left|\frac{S_{n_{k}}}{n_{k}}\left(\frac{n_{k}}{n_{k+1}} -1\right) +\frac{1}{n_{k+1}}\right|& =& \left|\frac{S_{n_{k}} +1}{n_{k+1}} - \frac{S_{n_{k}}}{n_{k}}\right| \to 0
\end{eqnarray*}
as $k \to \infty$.
Since $\frac{1}{n_{k+1}} \to 0$ as $k \to \infty$ and $\frac{S_{n_{k}}}{n_{k}} \to m_t(G_t)>0$ as $k \to \infty$ it follows that $\frac{n_{k}}{n_{k+1}} -1 \to 0$ as $k \to \infty$, in other words $\frac{n_{k}}{n_{k+1}}$ and $\frac{n_{k+1}}{n_k} \to 1$ as $k\to \infty$.
\end{proof}

By an easy modification of the above proof we can also obtain similar bounds on the projections of $\mu_t$.

\begin{lma} \label{egorov2}
For $m$ almost every $\mathbf{i} \in \Sigma$ there exists a choice of $t \in \{1,2\}$ such that the conclusion of Lemma \ref{egorov} holds with $\mu$ replaced by $\mu_t$ in (\ref{ld bound}).
\end{lma}

To prove Theorem \ref{main} it suffices to show that the local dimension of $\mu$ is what it should be at $x = \Pi(\mathbf{i})$ for  $\textbf{i} $ in a set of full $m$-measure. The proof will be split into two parts, concerning the lower and upper bound respectively.

\subsection{The lower bound}

Let $\mathbf{i} \in \Sigma$ belong to the set of full measure for which the conclusions of Propositions \ref{oseledets}, \ref{SMB} and Lemma \ref{egorov} hold simultaneously. In particular, let $t \in \{1,2\}$ be such that Lemma \ref{egorov} is satisfied for $m_t$. Write $\mathbf{x}=\Pi(\mathbf{i})$. 

Since $F$ satisfies the \emph{strong separation property}, there exists $\delta >0$, such that for any $\mathbf{j}= (j_1, j_2, \dots) \in \Sigma$ and $\mathbf{j}'= (j'_1, j'_2, \dots) \in \Sigma$ with $j_1 \neq j'_1$ we have $d(\Pi(\mathbf{j}), \Pi(\mathbf{j}')) \geq \delta$ where $d$ is the usual Euclidean metric. In other words, the components of $F$ are $\delta$-separated on the first level.

Consider the square $Q_2(\mathbf{x}, \delta \alpha_2(\mathbf{i}|_n))$. Observe that since any cylinder on the $n$th level which is distinct from $S_{\mathbf{i}|_n}([0,1]^2)$ must be at least $\delta \alpha_2(\mathbf{i}|_{n-1})$-separated from $S_{\mathbf{i}|_n}([0,1]^2)$ and therefore $Q_2(\mathbf{x}, \delta \alpha_2(\mathbf{i}|_n))$ only intersects the cylinder $S_{\mathbf{i}|_n}([0,1]^2)$. Therefore, it is easy to see that
$$Q_2(\mathbf{x}, \delta \alpha_2(\mathbf{i}|_n)) \cap F \subseteq B(\mathbf{x}, n, \alpha_2(\mathbf{i}|_n)).$$
By Lemma \ref{measure}, it follows that 
$$\mu(Q_2(\mathbf{x}, \delta \alpha_2(\mathbf{i}|_n))) \  \leq  \ m([\mathbf{i}|_n])\pi_s^{\mathbf{i},n}(\mu)\left(Q_1\left(\pi_s^{\mathbf{i},n}(\Pi(\sigma^n(\mathbf{i})), \frac{\alpha_2(\mathbf{i}|_n)}{\alpha_1(\mathbf{i}|_n)}\right)\right). $$

Fix $\epsilon >0$ and let $n_k$ be the subsequence from Lemma \ref{egorov}.  Observing that the sequence $\delta \alpha_2(\mathbf{i}|_{n_k})$ strictly decreases to zero, for any sufficiently small $r>0$ we can choose $k \in \mathbb{N}$ large enough such that 
\[
\delta \alpha_2(\mathbf{i}|_{n_{k+1}}) \leq r \leq \delta \alpha_2(\mathbf{i}|_{n_k}).
\]
 Assume $r>0$ is small enough to ensure $k \geq N_{\epsilon}$ and then we have
\begin{eqnarray*}
\frac{\log \mu\left( Q_2\left(\Pi(\mathbf{i}), r\right)\right)}{\log r} &\geq& \frac{\log \mu\left( Q_2\left(\Pi(\mathbf{i}), \delta\alpha_2(\mathbf{i}|_{n_k})\right)\right)}{\log \delta \alpha_2(\mathbf{i}|_{n_{k+1}})} \\ \\
&\geq& \frac{\log m([\mathbf{i}|_{n_k}]) + \log \pi_s^{\mathbf{i}, n_k}(\mu)\left(Q_1\left(\pi_s^{\mathbf{i}, n_k}(\Pi(\sigma^{n_k}\mathbf{i})), \frac{\alpha_2(\mathbf{i}|_{n_k})}{\alpha_1(\mathbf{i}|_{n_k})}\right)\right)}{\log \delta \alpha_2(\mathbf{i}|_{n_{k+1}})} \\ \\
&\geq& \frac{  \log m([\mathbf{i}|_{n_k}])  + (s_1- \epsilon) \log \frac{\alpha_2(\mathbf{i}|_{n_k})}{\alpha_1(\mathbf{i}|_{n_k})}}{\log \delta \alpha_2(\mathbf{i}|_{n_{k+1}})} \\ \\
&=& \frac{  -\frac{1}{n_k}\log m([\mathbf{i}|_{n_k}])-\frac{1}{n_k} (s_1 -  \epsilon) \log \frac{\alpha_2(\mathbf{i}|_{n_k})}{\alpha_1(\mathbf{i}|_{n_k})}}{- \frac{1}{n_k} \log \delta- \frac{1}{n_{k+1}}\frac{n_{k+1}}{n_k}\log \alpha_2(\mathbf{i}|_{n_{k+1}})} \\ \\
& \to & \frac{h(\mu) + (s_1-\varepsilon) (\chi_2(\mu)-\chi_1(\mu))}{\chi_2(\mu)}
\end{eqnarray*}
as $r \to 0$ ($k \to \infty$).  Finally, letting $\epsilon \to 0$ yields the desired lower bound.

\subsection{The upper bound}

Let $\mathbf{i} \in \Sigma$ belong to the set of full measure for which the conclusions of Propositions \ref{oseledets}, \ref{SMB} and Lemma \ref{egorov2} hold simultaneously. In particular, let $t \in \{1,2\}$ be such that Lemma \ref{egorov2} is satisfied for $m_t$. Let $n_k$ be the subsequence for $\mathbf{i}$ from Lemma \ref{egorov2}. Write $\mathbf{x}=\Pi(\mathbf{i})$. 

Consider the square $Q_2(\mathbf{x}, 2\alpha_2(\mathbf{i}|_{n_k}))$. Clearly 
$$B(\mathbf{x}, n_k, \alpha_2(\mathbf{i}|_{n_k})) \subseteq Q_2(\mathbf{x}, 2\alpha_2(\mathbf{i}|_{n_k})) \cap F $$
and therefore by Lemma \ref{measure2} it follows that 
$$\nu \circ \Pi^{-1}(Q_2(\mathbf{x}, 2 \alpha_2(\mathbf{i}|_{n_k})))  \ \geq  \  Cm_t \circ \tau([\mathbf{i}|_{n_k}])\pi_s^{\mathbf{i},n_k}(\mu_t)\left(Q_1\left(\pi_s^{\mathbf{i},n_k}(\Pi(\sigma^{n_k}(\mathbf{i})), \frac{\alpha_2(\mathbf{i}|_{n_k})}{\alpha_1(\mathbf{i}|_{n_k})}\right)\right). $$
Let $\epsilon>0$.  Consider small $r>0$, and since the sequence $2\alpha_2(\mathbf{i}|_{n_k})$  strictly decreases to zero we can choose $k$ such that
\[
2\alpha_2(\mathbf{i}|_{n_{k+1}}) \leq r \leq 2 \alpha_2(\mathbf{i}|_{n_k}).
\]
Assume $r$ is small enough to guarantee $k \geq N_{\epsilon}$.  Then by Lemmas \ref{measure2} and \ref{egorov2} we have
\begin{eqnarray*}
&\,& \hspace{-10mm} \frac{\log \nu \circ \Pi^{-1}\left( Q_2\left(\Pi(\mathbf{i}), r\right)\right)}{\log r}  \\ \\
&\leq& \frac{\log \nu \circ \Pi^{-1}\left( Q_2\left(\Pi(\mathbf{i}), 2\alpha_2(\mathbf{i}|_{n_{k+1}})\right)\right)}{\log 2\alpha_2(\mathbf{i}|_{n_{k}})} \\ \\
&\leq&   \frac{ \log Cm_t \circ \tau([\mathbf{i}|_{n_{k+1}}]) + \log \pi_s^{\mathbf{i}, n_{k+1}}(\mu_t)\left(Q_1\left(\pi_s^{\mathbf{i}, n_{k+1}}(\Pi(\sigma^{n_{k+1}}\mathbf{i})), \frac{\alpha_2(\mathbf{i}|_{n_{k+1}})}{\alpha_1(\mathbf{i}|_{n_{k+1}})}\right)\right)}{\log 2\alpha_2(\mathbf{i}|_{n_k})} \\ \\
&\leq& \frac{\log C+  \log m_t \circ \tau([\mathbf{i}|_{n_{k+1}}]) +( s_1+\epsilon) \log \frac{\alpha_2(\mathbf{i}|_{n_{k+1}})}{\alpha_1(\mathbf{i}|_{n_{k+1}})}}{\log 2 \alpha_2(\mathbf{i}|_{n_{k}})} \\ \\
&=& \frac{ -\frac{1}{n_{k+1}} \log C -\frac{1}{n_{k+1}}\log  \left( m_t \circ \tau([\mathbf{i}|_{n_{k+1}}]) \right) - \frac{1}{n_{k+1}} (s_1+ \epsilon)  \log \frac{\alpha_2(\mathbf{i}|_{n_{k+1}})}{\alpha_1(\mathbf{i}|_{n_{k+1}})}}{- \frac{1}{n_{k+1}} \log 2- \frac{1}{n_{k}}\frac{n_k}{n_{k+1}}\log \alpha_2(\mathbf{i}|_{n_{k}})} \\ \\
& \to & \frac{h(\mu) + (s_1+\varepsilon) (\chi_2(\mu)-\chi_1(\mu))}{\chi_2(\mu)}
\end{eqnarray*}
by Lemmas \ref{sequence1} and \ref{entropylemma} as $r \to 0$ ($k \to \infty$).  Finally,  letting $\epsilon \to 0$ yields the desired upper bound, and the result follows.\qed

\vspace{5mm}

\begin{centering}
\textbf{Acknowledgements} \\
This project grew out of  NJ's  Masters project at the University of Bristol, which was supervised by TJ.  A large part of the research was conducted whilst all three authors were in attendance at  the ICERM Semester Program on \emph{Dimension and Dynamics} in Spring 2016.  JMF was financially supported by a \emph{Leverhulme Trust Research Fellowship} (RF-2016-500). NJ was financially supported by the \emph{Leverhulme Trust} (RPG-2016-194). The authors thank Antti K\"aenm\"aki and Mike Todd for helpful comments.
\end{centering}

\end{document}